\documentclass[a4, 12pt]{amsart}
\usepackage[mathscr]{eucal}
\usepackage{amssymb}
\usepackage{latexsym}
\usepackage{amsthm}
\usepackage{color}
\theoremstyle{plain}
\newtheorem{theorem}{Theorem}[section]

\newtheorem{remark}{Remark}[section]
\newtheorem{lemma}{Lemma}[section]
\newtheorem{proposition}{Proposition}[section]

\makeatletter
\@addtoreset{equation}{section}

\title[Estimates on scalar curvature of self-shrinkers]
{Estimates on scalar curvature of self-shrinkers*}
\author [Q. -M. Cheng, F. Li  and G. Wei]{Qing-Ming Cheng, Fengjiang Li{$^+$}  and Guoxin Wei{$^+$}}
\address{Qing-Ming Cheng \newline
\indent
Mathematical Science Research Center, \newline 
\indent Chongqing University of Technology, Chongqing 400054, P. R. China. 
\vskip 1mm
\indent
Osaka Center Advanced Mathematical Institute \newline
\indent
Osaka Metropolitan University, Osaka 558-8585, Japan \newline
\indent qingmingcheng@yahoo.com, chengqingming@cqut.edu.cn }
\address{Fengjiang Li \newline
\indent
Mathematical Science Research Center, \newline 
\indent Chongqing University of Technology, Chongqing 400054, P. R. China. }
\address{Guoxin Wei \newline
\indent   School of Mathematical Sciences, South China Normal University,
\newline
\indent 510631, Guangzhou,  China, weiguoxin@tsinghua.org.cn}

\thanks{$+$ Fengjiang Li  and Guoxin Wei are corresponding authors}
\thanks{*This work was supported by  National Natural Science Foundation of China (Grant Nos.
12301062, 12571050, 12671065), Natural Science Foundation of Chongqing
(No.CSTB2024NSCQ-MSX0537), Japan Society for the Promotion of Science Grant-in-Aid
for Scientific Research (C) (Grant No. 25K06992).}

\begin{document}
\maketitle

\begin{abstract}
In this paper, we  study $n$-dimensional complete self-shrinkers in  $\mathbb R^{n+1}$
with constant squared norm $S$ of the second fundamental form. We partially resolve 
the conjecture on $n$-dimensional 
complete self-shrinkers in  $\mathbb R^{n+1}$ with constant squared norm $S$ of the second fundamental form.
Furthermore,  if the scalar curvature of an  $n$-dimensional  
self-shrinker is  constant, then we prove that the scalar curvature $R$ 
satisfies $R\leq n-1$.  We also classify $n$-dimensional complete self-shrinkers in  $\mathbb R^{n+1}$ 
with non-negative constant 
scalar curvature. 
\end{abstract}
\maketitle

\section{introduction}
\vskip2mm
\noindent
\noindent
This paper is concerned with study on the possible singularities of  mean curvature  flow, which is
one of  the most important problems in
the research  on the mean curvature flow. By making use of Huisken's monotonicity formula,  we  know that 
a solution to the flow is asymptotically self-similar near a given type I singularity. Thus, it
is modeled  by  self-shrinking solutions of the flow.
An  $n$-dimensional  hypersurface   $X: M\rightarrow \mathbb{R}^{n+1}$  in the $(n+1)$-dimensional
Euclidean space $\mathbb{R}^{n+1}$  is called {\it a self-shrinker} if it satisfies
\begin{equation*}
 H+ \langle X, N\rangle=0,
\end{equation*}
where  $N$  and $ H$ denote the unit normal vector  and mean curvature of  this
hypersurface.
Since self-shrinkers  describe all possible  blow-ups at a given singularity, self-shrinkers play an important role 
in the study on singularities  of the mean curvature flow.  \newline
 Throughout  this paper,  self-shrinkers, which we  consider,  are immersed self-shrinkers. We do not assume embedding condition
 except we need to emphasize  it.
  \newline
It is well-known that Abresch and Langer \cite{al}  classified closed self-shrinking curves in $\mathbb{R}^2$ 
and showed that the  round circle is the only embedded self-shrinker.
Huisken \cite{h2} proved  an  $n$-dimensional  compact  self-shrinker 
in $\mathbb{R}^{n+1}$ with  mean curvature $H\geq 0$ is isometric to the sphere $S^n(\sqrt{n})$. 
Furthermore,  Drugan \cite{dr} constructed  an immersed, non-embedded self-shrinker of {genus $0$}.
Hence, we know that  self-shrinkers
do not share common features with  the Hopf theorem: a topological sphere with constant mean curvature in $\mathbb R^3$ is the 
round sphere. But according to the  theorem of Brendle \cite{b}: if $X:M^2 \to \mathbb{R}^{3}$ is  
a 2-dimensional compact embedded  self-shrinker in $\mathbb R^3$
with genus $0$, then $X:M^2 \to \mathbb{R}^{3}$ is the round sphere.
Hence, self-shrinkers
share common features with {Alexandrov  theorem}  on {the embedded  sphere} 
with constant mean curvature in $\mathbb R^3$.
On the other hand, Huisken \cite{h3},  Colding and Minicozzi \cite{cm}  gave a complete classification for  $n$-dimensional  complete  
embedded self-shrinkers
in $\mathbb{R}^{n+1}$ with  mean curvature $H\geq 0$ and  with polynomial volume growth.
Furthermore,  Ding and Xin \cite{dx1} and  Cheng and Zhou \cite{cz} have proven that 
 an $n$-dimensional  complete  self-shrinker has polynomial volume growth if and only if  it is proper.
 On the other hand, it is also known that there exist complete self-shrinkers without polynomial volume growth 
 in $\mathbb R^{n+1}$ in Halldorsson \cite{h}. \\
 Since  many formulas on self-shrinkers are very similar to formulas 
on minimal hypersurfaces in the unit sphere in the some sense, one hopes
that self-shrinkers share some common properties of  minimal 
hypersurfaces in the unit sphere.
It is well-known, that for  minimal hypersurfaces in the unit sphere, the following Chern problems are 
very important:
\vskip2mm
\noindent
{\bf Chern problems}.
For $n$-dimensional compact minimal  hypersurfaces in $S^{n+1}(1)$
with constant  squared norm $S$ of the second fundamental  form, is the following true?
\begin{enumerate}
\item $S\leq c(n)$, where $c(n)$ is a constant depending only on dimension $n$,
\item the values of $S$ of the squared norm
of the second fundamental form  are discrete, 
\item the values of $S$
should determine the hypersurfaces up to a rigid motion in the ambient sphere
$S^{n+1}(1)$.
\end{enumerate}
\vskip2mm
\noindent
For self-shrinkers, the 
following conjecture is well-known:
\vskip2mm
\noindent
{\bf Conjecture}.  An $n$-dimensional complete self-shrinker  $X: M\rightarrow \mathbb{R}^{n+1}$
 with constant  squared norm of the second fundamental form
is isometric to one of  
\begin{enumerate}
\item  $S^{n}(\sqrt{n})$,
\item $\mathbb{R}^{n}$, 
\item
$S^k (\sqrt k)\times \mathbb{R}^{n-k}$, $1\leq k\leq n-1$.
\end{enumerate}
\vskip2mm
 \noindent
Cheng and Ogata \cite{co}  confirmed this conjecture for $n=2$. Namely, they have  proven
the following:

\begin{theorem} A $2$-dimensional complete self-shrinker  $X: M\rightarrow \mathbb{R}^{3}$
 with constant  squared norm of the second fundamental form
is isometric to one of  
\begin{enumerate}
\item  $S^{2}(\sqrt{2})$,
\item $\mathbb{R}^{2}$, 
\item
$S^1 (1)\times \mathbb{R}$.
\end{enumerate}
\end{theorem}
 
\begin{remark}  For $n\geq 3$, the above conjecture is still open. 
As a partial result, for $n=3$, Cheng, Li and Wei \cite{cliw, cliw1} have resolved the conjecture under the condition
that $f_3=\sum_i\lambda_i^3$  or $f_4=\sum_i\lambda_i^4 $ is constant, where $\lambda_i$ denotes the principal curvature.
For general $n$, Cheng and Wei \cite{cw} and Cheng, Wei and Yano \cite{cwy} have also obtained partial results.
\end{remark}
\noindent
We  will resolve  the above conjecture,  affirmatively,  if the scalar curvature $R$ satisfies $R\geq -\dfrac{7}{5}$. 
\begin{theorem}
An $n$-dimensional complete self-shrinker  $X: M\rightarrow \mathbb{R}^{n+1}$
 with constant  squared norm of the second fundamental form
is isometric to one of  
\begin{enumerate}
\item  $S^{n}(\sqrt{n})$,
\item $\mathbb{R}^{n}$, 
\item
$S^k (\sqrt k)\times \mathbb{R}^{n-k}$, $1\leq k\leq n-1$.
\end{enumerate}
if the scalar curvature $R$ satisfies $R\geq -\dfrac{7}{5}$.
\end{theorem}
\begin{remark} Since we do not assume the condition of polynomial volume growth, we can not make use of 
the integral formulas and Stokes theorem as in Colding and Minicozzi \cite{cm} and Cheng and Wei \cite{cw}.
In order to overcome  these substantial difficulties, we need to use the generalized maximum principle due to Cheng and Peng
\cite{cp}.  Thus, we  must  give precise estimates for several inequalities which include the second fundamental form 
and its covariant  derivatives  in order to prove our  theorem (see theorem 3.1 in section 3). 
\end{remark}
\noindent
For minimal hypersurfaces in the unit sphere, we know that the scalar curvature is constant
if and only if the squared norm  of the second fundamental form is constant, thanks to
the Gauss equation. But self-shrinkers do not share this property. 
In \cite{g}, Guo proved that compact self-shrinkers with constant scalar curvature in $\mathbb R^{n+1}$
are isometric to the sphere $S^{n}(\sqrt n)$. Since an $n$-dimensional compact self-shrinker  
in $\mathbb R^{n+1}$ must have a convex point, the constant scalar curvature must be positive at this point.
By making use of Stokes formula, Guo \cite{g} proved  the scalar curvature $R= (n-1)$.  Thus, 
Chern type problems on compact self-shrinkers with constant scalar curvature were resolved by Guo. On the other hand, 
the study on $n$-dimensional complete non-compact self-shrinkers in $\mathbb R^{n+1}$ is  
more important and  in this case, one can not use  integral formulas and Stokes theorem. 
This will be a substantial problem, which we should  overcome.
Luo, Sun and Yin \cite{lsy} have proven that an 
$n$-dimensional complete  self-shrinker in $\mathbb R^{n+1}$  with polynomial volume growth and
positive constant scalar curvature  is isometric to one of
\begin{enumerate}
\item $S^{k}(\sqrt k)\times \mathbb{R}^{n-k}, \ 2\leq k\leq n-1$,
\item $S^{n}(\sqrt{n})$.
\end{enumerate}
In fact, according to the Gauss equation, if  the scalar curvature $R$ is positive, we have 
$H\neq 0$ because of $H^2-S=R$. Hence, according to the results and proof 
due to Colding and Minicozzi \cite{cm},
we can remove the condition that scalar curvature is constant, thanks to the Gauss equation.
\begin{proposition} 
An $n$-dimensional complete  self-shrinker in $\mathbb R^{n+1}$  with polynomial volume growth and
positive  scalar curvature  is isometric to one of
\begin{enumerate}
\item $S^{n}(\sqrt{n})$,
\item $S^{k}(\sqrt k)\times \mathbb{R}^{n-k}, \ 2\leq k\leq n-1$.
\end{enumerate}
\end{proposition} 
\begin{proof}
Since the scalar curvature $R$ is positive, we can assume $H>0$ according to the Gauss equation
$H^2-S=R$. Therefore, by applying  the results in the theorem 0.17 (theorem 10.1) of Colding 
and Minicozzi \cite{cm}, we finish our proof. We should notice that in the theorem 0.17 of Colding and 
Minicozzi, they assume the embedding condition.  But this embedding condition is  only used 
in the  case 2 of the step 3 in their proof, that is, at some point, the rank of the second fundamental form
is one. In this case, we know the scalar curvature $R=0$. Since the scalar curvature is positive, 
this case does not occur. Thus, we can make use of the results of Colding and Minicozzi  \cite{cm}.
\end{proof}
\noindent
Furthermore,  we will prove the following:
\begin{theorem}
An $n$-dimensional  self-shrinker in $\mathbb R^{n+1}$  with 
positive constant scalar curvature $R$ satisfies $0<R\leq n-1$ and $S\leq 1$.
\end{theorem}
\begin{remark} Since the sphere $S^n(\sqrt n)$ satisfies $R=n-1$ and $S=1$,  our estimates
are  optimal.
\end{remark}
\noindent
By making use of the generalized maximum principle due to Cheng and Peng \cite{cp}, we obtain the following:
\begin{theorem} 
An $n$-dimensional complete  self-shrinker in $\mathbb R^{n+1}$  with 
non-negative constant scalar curvature  either is isometric to  one of
\begin{enumerate}
\item $S^{n}(\sqrt{n})$,
\item $\mathbb {R}^{n}$,
\item $S^{k}(\sqrt k)\times \mathbb{R}^{n-k}, \ 1\leq k\leq n-1$,
\item $\Gamma \times \mathbb{R}^{n-1}$, where $\Gamma$ is a  complete self-shrinker curve in $\mathbf R^2$,
\end{enumerate}
or satisfies $0<R<n-2$, $\frac{R}{n-1}\leq S<1$, $\sup S=1$, $\frac{n}{n-1}R\leq H^2$,
$|X|^2\geq \frac{n}{n-1}R$ and $\sup |X|^2=\infty$.
\end{theorem}

\begin{remark} We think that there do not exist complete self-shrinkers with constant scalar curvature 
such that they  satisfy $0<R< n-2$, $\frac{R}{n-1}\leq S<1$, $\sup S=1$, $\frac{n}{n-1}R\leq H^2<R+1$,
$|X|^2\geq \frac{n}{n-1}R$ and $\sup |X|^2=\infty$.
\end{remark}
\vskip1mm
\noindent
Furthermore, we get a pinching theorem on the mean curvature of complete self-shrinkers with non-negative scalar curvature.
\begin{theorem} 
For an $n$-dimensional complete  self-shrinker $X: M\rightarrow \mathbb{R}^{n+1}$, $n\geq 3$,  in $\mathbb R^{n+1}$  with 
non-negative constant scalar curvature, if the mean curvature $H$ satisfies
$$
H^2\leq \dfrac{n-1}{n-2}R,
$$
then   $X: M\rightarrow \mathbb{R}^{n+1}$  is isometric to  one of
\begin{enumerate}
\item $S^{n}(\sqrt{n})$,
\item $\mathbb {R}^{n}$,
\item $S^{n-1}(\sqrt {n-1})\times \mathbb{R}$.
\end{enumerate}
\end{theorem}

\vskip5mm
\section {Preliminaries}
\vskip2mm

\noindent
Let $X: M\rightarrow\mathbb{R}^{n+1}$ be an
$n$-dimensional connected hypersurface in the $(n+1)$-dimensional Euclidean space
$\mathbb{R}^{n+1}$. We choose a local orthonormal frame field
$\{e_A\}_{A=1}^{n+1}$ in $\mathbb{R}^{n+1}$ with dual coframe field
$\{\omega_A\}_{A=1}^{n+1}$, such that, restricted to $M$,
$e_1,\cdots, e_n$ are tangent to $M^n$.
Then we have
\begin{equation*}
dX=\sum_i\limits \omega_i e_i, \  \
de_i=\sum_j\limits \omega_{ij}e_j+\omega_{i n+1}e_{n+1}, \ 
de_{n+1}=\sum_i\omega_{n+1 i}e_i,
\end{equation*}
where $\omega_{ij}$ is the Levi-Civita connection of  $X: M\rightarrow\mathbb{R}^{n+1}$.
Because of $\omega_{n+1}=0$ along $M$, one has 
\begin{equation}\label{2.1-2}
\omega_{in+1}=\sum_j h_{ij}\omega_j,\quad
h_{ij}=h_{ji}.
\end{equation}
$$
H= \sum_i\limits h_{ii}, \  \ \vec{h}=\sum_{i,j}h_{ij}\omega_i\otimes\omega_je_{n+1}
$$
are called {\it the mean curvature} and {\it the second fundamental  form}, respectively.
Setting $S=\sum_{i,j}\limits (h_{ij})^2$, components $R_{ijkl}$ of the curvature tensor, components $R_{ij}$
of the Ricci curvature tensor  and the scalar curvature $R$ are given by
\begin{equation}\label{eq:2.2}
R_{ijkl}=h_{ik}h_{jl}-h_{il}h_{jk}, \quad R_{ij}=Hh_{ij}-\sum_kh_{ik}h_{kj}, \quad R=H^2-S.
\end{equation}
\noindent
Defining the
covariant derivative of $h_{ij}$ by
\begin{equation*}
\sum_{k}h_{ijk}\omega_k=dh_{ij}+\sum_kh_{ik}\omega_{kj}
+\sum_k h_{kj}\omega_{ki},
\end{equation*}
we obtain the Codazzi equations
\begin{equation}\label{eq:2.3}
h_{ijk}=h_{ikj}.
\end{equation}
Defining
\begin{equation*}
\sum_lh_{ijkl}\omega_l=dh_{ijk}+\sum_lh_{ljk}\omega_{li}
+\sum_lh_{ilk}\omega_{lj}+\sum_l h_{ijl}\omega_{lk},
\end{equation*}
we have the following Ricci identities:
\begin{equation}\label{eq:2.4}
h_{ijkl}-h_{ijlk}=\sum_m
h_{mj}R_{mikl}+\sum_m h_{im}R_{mjkl}.
\end{equation}
For a smooth function $f$, the $\mathcal{L}$-operator is defined by
\begin{equation*}
\mathcal{L}f=\Delta f-\langle X,\nabla f\rangle,
\end{equation*}
where $\Delta$ and $\nabla$ denote the Laplacian and the gradient
operator, respectively.
For functions $f_3=\sum_{i,j,k}h_{ij}h_{jk}h_{ki}=\sum_i\lambda_i^3$ and
$f_4=\sum_{i,j,k,l}h_{ij}h_{jk}h_{kl}h_{li}=\sum_i\lambda_i^4$, where $\lambda_i$ denotes  the principal curvature,
we have 
\begin{equation}\label{eq:2.5}
 \begin{aligned}
 & \dfrac13\mathcal{L}f_3=(1-S)f_3+2\sum_{i,j,k}\lambda_ih_{ijk}^2,\\
&  \dfrac14\mathcal{L}f_4= (1-S)f_4+ 2A+B,
\end{aligned}
 \end{equation}
 where $ A = \sum_{i,j,k} \lambda_i^2 h_{ijk}^2$ and  $B = \sum_{i,j,k} \lambda_i \lambda_j h_{ijk}^2$.
\vskip1mm
\noindent
By a direct calculation, we can derive the following  formulas, which  can also  be found in  \cite{h2}, \cite{cm}, \cite{cw}.
\begin{lemma}
For  an $n$-dimensional  self-shrinker $X:M^n\rightarrow \mathbb{R}^{n+1}$  in $\mathbb R^{n+1}$,  we know
\begin{equation}\label{eq:2.6}
\dfrac{1}{2}\mathcal{L}
|X|^{2}=n-|X|^{2}, \quad \mathcal{L}H=H(1-S),
\end{equation}
\begin{equation}\label{eq:2.7} 
\dfrac{1}{2}\mathcal{L}S
=\sum_{i,j,k}h_{ijk}^{2}+(1-S)S,  \ \  \dfrac12\mathcal{L}H^2=|\nabla H|^2+H^2(1-S).
\end{equation}
If  $H>0$, we have
\begin{equation}\label{eq:2.8}
\begin{aligned}
\mathcal{L}\dfrac {1} {H^2}
&=-\dfrac{2(1-S)}{H^2}+\dfrac{6}{H^4}|\nabla H|^2,
\end{aligned}
\end{equation}
\begin{equation}\label{eq:2.9}
\begin{aligned}
\dfrac12\mathcal{L}\dfrac {S} {H^2}
&=\dfrac{1}{H^4}\sum_{i, j, k}|h_{ij}\nabla_kH-h_{ijk}H|^2-\dfrac1H\langle\nabla H,\nabla \dfrac {S}{H^2}\rangle.
\end{aligned}
\end{equation}
\end{lemma}
\noindent
The following generalized maximum principle for $\mathcal L$-operator due to Cheng and Peng \cite{cp} will play an important role.
\vskip2mm
\noindent
{\bf Generalized maximum principle for $\mathcal L$-operator}. 
Let $X : M^{n}\to \mathbb{R}^{n+p}$ be a complete self-shrinker with Ricci curvature bounded from below.
Let $f$ be any $C^{2}$-function bounded from above on this self-shrinker.
Then, there exists a sequence of points $\{p_{m}\}\subset M^{n}$, such that
\begin{equation*}
\lim_{m\rightarrow\infty} f(p_{m})=\sup f,\quad
\lim_{m\rightarrow\infty} |\nabla f|(p_{m})=0,\quad
\limsup_{m\rightarrow\infty}\mathcal{L}f(p_{m})\leq 0.
\end{equation*}
 \vskip10mm
 
 \section{Proof of  theorem 1.2}
\vskip2mm
\noindent
In order to prove our theorem 1.2, we need to remove the condition of polynomial volume growth 
in the theorem 1.1 of Cheng and Wei \cite{cw} as follows.
\begin{theorem}\label{thm3-1}
An $n$-dimensional complete self-shrinker  $X: M\rightarrow \mathbb{R}^{n+1}$
 with constant  squared norm of the second fundamental form
is isometric to one of  
\begin{enumerate}
\item  $S^{n}(\sqrt{n})$,
\item $\mathbb{R}^{n}$, 
\item
$S^k (\sqrt k)\times \mathbb{R}^{n-k}$, $1\leq k\leq n-1$.
\end{enumerate}
if  $S\leq \dfrac{7}{5}$.
\end{theorem}
\begin{remark} Outline of proof of theorem \ref{thm3-1}. In  the theorem \ref{thm3-1}, we do not assume 
that self-shrinker  $X: M\rightarrow \mathbb{R}^{n+1}$ has polynomial volume growth. We can not use the integral formulas 
and Stokes theorem as in \cite{cm} and 
\cite{cw}. Since $S$ is constant, according to the Gauss equation, we know that the Ricci curvature is bounded
and functions $f_3=\sum_{i}\lambda_i^3$  and $f_4=\sum_{i}\lambda_i^4$ are bounded, where $\lambda_i$ denotes the  principal curvature.
In  step 1, we want to prove $S>\frac65$ if $S>1$.
We  apply the generalized maximum principle to a function $F=\dfrac14Sf_4-\dfrac 16 c f_3^2$  for a given constant $c$.
For our purposes, we will take different values of $c$ in the different cases. In order to prove our assertion, we need to 
symmetrize $h_{ijkl}$ by using $u_{ijkl}=\frac14(h_{ijkl}+h_{jkli}+h_{klij}+h_{lijk})$. In step 2, we need to prove $S>\frac75$ if $S>\frac65$.
In this case, we must do more precise estimates for $u_{ijkl}$, which is  a very hard work. Thus, we obtain $S\leq 1$ if $S\leq \frac75$.
Since $S$ is constant, in view of (\ref{eq:2.7}), we have  $h_{ijk}\equiv 0$.  Thus, we know that the second fundamental form of 
self-shrinker  $X: M\rightarrow \mathbb{R}^{n+1}$ is parallel. According to the results due to Lawson \cite{l},  we finish our proof.
\end{remark}
\noindent
We will  use the same notation as in \cite{cw}.
The following pointwise estimates can be found in Cheng and Wei \cite{cw}. 
Define  $\lambda_1=\max\limits_i\{\lambda_i\} $ and  $\lambda_2=\min\limits_i\{\lambda_i\}$ at each point.
\begin{lemma}
\begin{equation}\label{eq:3.1}
\begin{aligned}
&f=f_4-\frac{f_3^2}{S}\geq \frac{(\lambda_1-\lambda_2)^2}{\lambda_1^2+\lambda_2^2}(\lambda_1\lambda_2)^2,\\
&A-B \le \frac 13(\lambda_1-\lambda_2)^2S(S-1)(1-\alpha),\\
&
 \sum_k\biggl(\sum_i\lambda_i^2h_{iik}\biggl)^2\le \frac{1+2\alpha}3S(S-1)f,\\
 &\biggl(\sum_{i,j,k}\lambda_ih_{ijk}^2\biggl)^2 \le \biggl[\frac 13(A+2B)
-\frac {4}3\sum_k\frac 1{S+2\lambda_k^2}\left( \sum_i
\lambda_i^2h_{iik}\right)^2\biggl]S(S-1),
\end{aligned}
\end{equation}
where  $\alpha=\dfrac{\sum_ih_{iii}^2 }{tS^2}$. 
\end{lemma}
\noindent
Since $S$ is constant, we have
\begin{equation}\label{eq:3.2}
\begin{aligned}
&\sum_{i,j,k}h_{ijk}^2= S(S-1),\ \ \mathcal L h_{ij}=(1-S)h_{ij},\\
&\sum_{i,j,k,l}h_{ijkl}^2=S(S-1)(S-2) +3(A-2B).
\end{aligned}
\end{equation}
Defining $u_{ijkl}$ by
$$
  u_{ijkl}:=\frac 14(h_{ijkl}+h_{jkli}+h_{klij}+h_{lijk} ),
  $$
 in view of the Ricci identity (\ref{eq:2.4}),  a direct computation yields
 \begin{equation}\label{eq:3.3}
\begin{aligned}
  \sum_{i,j,k,l}h_{ijkl}^2 = \sum_{i,j,k,l}u_{ijkl}^2 +\frac32\bigl( Sf_4-f_3^2\bigl).
\end{aligned}
\end{equation}
If  $S> 1$,  defining $S-1=tS$, we  know $0<t\leq \dfrac27$ for  $S\leq \frac75$. 
Thus,  we obtain the  following:

\begin{proposition}\label{pro3.1} For an $n$-dimensional  self-shrinker  $X: M\rightarrow \mathbb{R}^{n+1}$
 with constant  squared norm of the second fundamental form and $S>1$, we have
\begin{equation}\label{eq:3.4}
\begin{aligned}
&S(S-1)(S-2) +3(A-2B) \\
&\ge 2 Sf-2A+\frac 4{S(S-1)}\sum_i\lambda_i^2\big (\sum_j\lambda_j^2h_{jji}\big)^2\\
&+\frac {2f_3}S \sum_{i,j,k}\lambda_ih_{ijk}^2  +\frac 1{S^2}\big(\sum_{i,j,k}\lambda_ih_{ijk}^2 \big)^2.
\end{aligned}
\end{equation}
\end{proposition}

\begin{proof} In order to give a more precise estimate for  $u_{ijkl}$, we define $b_{ij}$ by 
$$
b_{ij}=(\lambda_i^2-\dfrac{f_3}{S}\lambda_i)\delta_{ij}.
$$
Direct computations yield
\begin{equation*}
\begin{aligned}
&\sum_{i,j}h_{ij}b_{ij}=0, \ \ \sum_{i,j}b_{ij}^2=f_4-\dfrac{f_3^2}{S}=f.\\
\end{aligned}
\end{equation*}
Since $S$ is constant, we have
\begin{equation*}
\begin{aligned}
&\sum_{i,j}h_{ij}h_{ijkl}+\sum_{i,j}h_{ijk}h_{ijl}=0,\quad \sum_{i,j,k}h_{ijk}h_{ijkl}=0.\\
\end{aligned}
\end{equation*}
Hence, we get
\begin{equation*}
\begin{aligned}
&\sum_{i,j,k,l}h_{ij}h_{ijkl}h_{kl}=-\sum_{i,j,k}\lambda_ih_{ijk}^2,\\
&\sum_{i,j}\lambda_ih_{iijj}\lambda_j^2=\sum_{i,j,k,l}h_{ij}h_{ijkl}\lambda_k\lambda_l\delta_{kl}=-A,\\
&\sum_{i,j}\lambda_i^2h_{iijj}\lambda_j=\sum_{i,j}\lambda_i^2h_{iijj}\lambda_j-\sum_{i,j}\lambda_ih_{iijj}\lambda_j^2-A\\
&=\sum_{i,j}\lambda_i^2\lambda_j(h_{iijj}-h_{jjii})-A=-A+Sf.
\end{aligned}
\end{equation*}
Since $u_{ijkl}$ are symmetric in $i, j, k, l$, we have
\begin{equation*}
\begin{aligned}
&\sum_{i,j,k,l}u_{ijkl}h_{ij}b_{kl}=-A+\dfrac12Sf+\dfrac{f_3}S\sum_{i,j,k}\lambda_ih_{ijk}^2,\\
&\sum_{i,j,k}u_{ijkl}h_{ijk}=-\dfrac{3}2\sum_i\lambda_i^2h_{iil}\lambda_l,\\
&\sum_{i,j,k}b_{ij}h_{kl}h_{ijk}=\sum_{i,j}\lambda_i^2h_{iij}h_{jl}.\\
\end{aligned}
\end{equation*}
For any $\delta$, $\beta_i$ and $\gamma $, we know 
\begin{equation*}
\begin{aligned}
&\sum_{i,j,k,l}\bigl\{u_{ijkl}+\delta (h_{ij}b_{kl}+h_{kl}b_{ij})+(\beta_ih_{jkl}+\beta_jh_{ikl}+\beta_kh_{ijl}+\beta_lh_{ijk})
+\gamma h_{ij}h_{kl}\bigl\}^2\geq 0.
\end{aligned}
\end{equation*}
If $f=0$, since $2Sf=\sum_{i,j}\lambda_i^2\lambda_j^2(\lambda_i-\lambda_j)^2$, we know 
all  non-zero principal curvatures are equal.  From the third inequality in lemma 3.1, we have
$$
\sum_i\bigl(\sum_{j}\lambda_j^2h_{jji}\bigl)^2=0.
$$
In this case,  by taking 
$$
\gamma=\dfrac{1}{S^2}\sum_{i,j,k}\lambda_ih_{ijk}^2, \quad \delta=-1, \quad \beta_i=0,
$$
we get 
\begin{equation*}
\begin{aligned}
&\sum_{i,j,k,l}u_{ijkl}^2
&\ge -2A+\frac {2f_3}S \sum_{i,j,k}\lambda_ih_{ijk}^2  +\frac 1{S^2}\big(\sum_{i,j,k}\lambda_ih_{ijk}^2 \big)^2.
\end{aligned}
\end{equation*}
The inequality (3.4) is proved.\\
Next, we  consider $f>0$. 
By making use of the above formulas and taking 
\begin{equation*}
\begin{aligned}
&\gamma=\dfrac{1}{S^2}\sum_{i,j,k}\lambda_ih_{ijk}^2,  \quad \beta_i=\dfrac{-2\delta+3}{8{S(S-1)}}\sum_{j}\lambda_j^2h_{jji}\lambda_i,\\
&\delta=-{\dfrac{2S(S-1)\bigl(-A+\dfrac12Sf+\dfrac{f_3}S\sum_{i,j,k}\lambda_ih_{ijk}^2+\dfrac 3{4S(S-1)}\sum_i\lambda_i^2\big (\sum_j\lambda_j^2h_{jji}\big)^2\bigl)}{2S^2(S-1)f-\sum_j\lambda_j^2(\sum_i\lambda_i^2h_{iij})^2}},
\end{aligned}
\end{equation*}
a long and direct computation yields
\begin{equation*}
\begin{aligned}
&\sum_{i,j,k,l}u_{ijkl}^2
\geq \dfrac{\biggl(-2A+Sf+\dfrac{2f_3}S\sum_{i,j,k}\lambda_ih_{ijk}^2
+\dfrac 3{2S(S-1)}\sum_i\lambda_i^2\big (\sum_j\lambda_j^2h_{jji}\big)^2\biggl)^2}
{2Sf-\dfrac{\sum_j\lambda_j^2(\sum_i\lambda_i^2h_{iij})^2}{S(S-1)}}\\
&+\frac 1{S^2}\big(\sum_{i,j,k}\lambda_ih_{ijk}^2 \big )^2+\dfrac 9{4S(S-1)}\sum_i\lambda_i^2\big (\sum_j\lambda_j^2h_{jji}\big)^2\\
&\ge \frac 12 Sf-2A
+\frac 4{S(S-1)}\sum_i\lambda_i^2\big(\sum_j\lambda_j^2h_{jji}
\big )^2+\frac {2f_3}S \sum_{i,j,k}\lambda_ih_{ijk}^2  +\frac 1{S^2}\big(\sum_{i,j,k}\lambda_ih_{ijk}^2 \big)^2,
\end{aligned}
\end{equation*}
where in  the last inequality, we used $(x+y)^2\geq 4xy$ with
$$
\begin{aligned}
&x=-2A+\dfrac12Sf+\dfrac{2f_3}S\sum_{i,j,k}\lambda_ih_{ijk}^2
+\dfrac 7{4S(S-1)}\sum_i\lambda_i^2\big (\sum_j\lambda_j^2h_{jji}\bigl)^2,\\
&y=\dfrac12Sf-\dfrac 1{4S(S-1)}\sum_i\lambda_i^2\big (\sum_j\lambda_j^2h_{jji}\big)^2.
\end{aligned}
$$
From (\ref{eq:3.2}) and  (\ref{eq:3.3}) and the above inequality,  we have
\begin{equation*}
\begin{aligned}
&S(S-1)(S-2) +3(A-2B)\\
&\geq  2Sf-2A+\frac 4{S(S-1)}\sum_i\lambda_i^2\bigl (\sum_j\lambda_j^2h_{jji}\bigl)^2
+\frac {2f_3}S \sum_{i,j,k}\lambda_ih_{ijk}^2  +\frac 1{S^2}\bigl(\sum_{i,j,k}\lambda_ih_{ijk}^2 \bigl)^2.\\
\end{aligned}
\end{equation*}
\end{proof}

\vskip2mm
\noindent
{\it Proof of theorem 3.1}. 
Since $S$ is constant, we know that Ricci curvature is bounded from the Gauss equation
and the principal curvatures $\lambda_i$, $h_{ij}$, $h_{ijk}$ and  $h_{ijkl}$ for any $i, j, k, l$ 
are bounded according to the formula (3.2). 
Thus,  it is obvious  that the function $F=\dfrac14Sf_4-\dfrac 16 c f_3^2$  for a given constant $c$ is  bounded
and in view of (\ref{eq:2.5}), we have
\begin{equation}\label{eq:3.5}
\begin{aligned}
\mathcal LF&= S(1-S)f_4+S(2A+B)\\
&-c\bigl((1-S)f_3^2+2f_3\sum_{i,j,k}\lambda_ih_{ijk}^2 +3\sum_{j}(\sum_i\lambda_i^2h_{iij})^2\bigl).
\end{aligned}
\end{equation}
By  applying  the generalized maximum principle for $\mathcal L$-operator to the function 
$F$, there exists a sequence $\{p_m\} \subset M$ such that 
 $$
 \lim_{m\to\infty}F(p_m)=\sup F,
\ \ \lim_{m\to\infty}|\nabla F|(p_m)=0,
$$
$$
  \lim\sup_{m\to\infty} \mathcal LF(p_m)\leq 0.
$$
Since $\lambda_i$,  $h_{ijk}$ and  $h_{ijkl}$  are bounded, we can assume that 
 $\{\lambda_i(p_m)\}$, $\{h_{ijk}(p_m)\}$ and  $\{h_{ijkl}(p_m)\}$  converge 
 if necessary  by taking a subsequence of $\{p_m\}$. By  abusing  notation, we still use 
 $\lambda_i$,  $h_{ijk}$ and  $h_{ijkl}$ to denote
 limits of $\{\lambda_i(p_m)\}$, $\{h_{ijk}(p_m)\}$ and  $\{h_{ijkl}(p_m)\}$, respectively and all computations are processed under limits.
 Hence, we have
\begin{equation}\label{eq:3.6}
 -(S-1)\bigl(Sf_4-cf_3^2\bigl)\leq 2cf_3\sum_{i,j,k}\lambda_ih_{ijk}^2- S(2A+B)+3c\sum_{j}(\sum_i\lambda_i^2h_{iij})^2
\end{equation}
from (\ref{eq:3.5}) and $\lim\sup_{m\to\infty} \mathcal LF(p_m)\leq 0$.
If $S\leq 1$, since $S$ is constant, in view of (\ref{eq:2.7}), we have  $h_{ijk}\equiv 0$.  Hence,  the second fundamental form of 
self-shrinker  $X: M\rightarrow \mathbb{R}^{n+1}$ is parallel. According to the results due to Lawson \cite{l},  we know that 
the self-shrinker  $X: M\rightarrow \mathbb{R}^{n+1}$
is isometric to one of  $S^{n}(\sqrt{n})$,  $\mathbb{R}^{n}$, 
$S^k (\sqrt k)\times \mathbb{R}^{n-k}$, $1\leq k\leq n-1$. Next we will prove that  $1<S\leq \dfrac75$ does not occur. Namely, we will prove $S>\frac75$ if $S>1$. This is a contradiction since $S$ is constant.\\
Taking $c=2$ in (\ref{eq:3.6}), we obtain
$$
 0\leq (S-1)\bigl(\dfrac12Sf_4-f_3^2\bigl)+2\sum_{i,j,k}\lambda_ih_{ijk}^2 f_3-\dfrac S2(2A+B)+3\sum_{j}(\sum_i\lambda_i^2h_{iij})^2.
 $$
 Therefore,  we can make use of  the same  arguments (word by word)  as in the part I of proof of the theorem 1.1 of 
 Cheng and Wei \cite{cw} (page 4903)  in place of their lemma 3.1 by the above inequality
 to prove  $S>\dfrac65$ if $S>1$. \\
From now, we prove $S>\dfrac75$ if $S>\dfrac65$. 
We define a function $\phi$ by 
\begin{equation}
\phi=-\dfrac{13}{3}tSf-\frac {2f_3}S \sum_{i,j,k}\lambda_ih_{ijk}^2  -\frac 1{S^2}\bigl(\sum_{i,j,k}\lambda_ih_{ijk}^2 \bigl)^2
-\dfrac{4}{tS^2}\sum_{j}\lambda_j^2(\sum_i\lambda_i^2h_{iij})^2.\\
\end{equation}
Letting $c=1+c_1$ in (\ref{eq:3.6}) with $c_1=\dfrac{22+3t-\sqrt{(22+3t)^2-400}}{26}$, we infer
\begin{equation}\label{eq:3.8}
(\dfrac{10}{13}+c_1)^2=\dfrac{3(14+t)}{13}c_1.
\end{equation}
From (\ref{eq:3.6}) and the definition of $\phi$, we obtain
\begin{equation*}
\begin{aligned}
\phi&\leq -\dfrac{13}{3}(2A+B)-\dfrac{13c_1t}{3}f_3^2+(\dfrac{10}3+\dfrac{13}3c_1)\frac {2f_3}S \sum_{i,j,k}\lambda_ih_{ijk}^2  \\
&+\dfrac{13(1+c_1)}S\sum_{j}(\sum_i\lambda_i^2h_{iij})^2-\frac 1{S^2}\bigl(\sum_{i,j,k}\lambda_ih_{ijk}^2 \bigl)^2
-\dfrac{4}{tS^2}\sum_{j}\lambda_j^2(\sum_i\lambda_i^2h_{iij})^2\\
&\leq -\dfrac{13}{3}(2A+B)+\bigl\{\dfrac{13}{3tc_1}(\dfrac{10}{13}+c_1)^2-1\bigl\} \dfrac{1}{S^2}\bigl(\sum_{i,j,k}\lambda_ih_{ijk}^2\bigl)^2  \\
&+\dfrac{13(1+c_1)}S\sum_{j}(\sum_i\lambda_i^2h_{iij})^2
-\dfrac{4}{tS^2}\sum_{j}\lambda_j^2(\sum_i\lambda_i^2h_{iij})^2,\\
\end{aligned}
\end{equation*}
where we used $2xy\leq  x^2+y^2$ in the last inequality
with
$$
\begin{aligned}
&x=\sqrt{\dfrac{3}{13tc_1}}(\dfrac{10}{3}+\frac{13c_1}3)\dfrac1S\sum_{i,j,k}\lambda_ih_{ijk}^2, \ \
y=\sqrt{\dfrac{13tc_1}{3}}f_3.
\end{aligned}
$$
 Applying the inequality  in  the lemma 3.1 for $\bigl(\sum_{i,j,k}\lambda_ih_{ijk}^2 \bigl)^2$ 
to the above inequality about $\phi$, we get 
 \begin{equation*}
\begin{aligned}
\phi&\leq -\dfrac{13}{3}(2A+B)+\dfrac{13(1+c_1)}S\sum_{j}(\sum_i\lambda_i^2h_{iij})^2
-\dfrac{4}{tS^2}\sum_{j}\lambda_j^2(\sum_i\lambda_i^2h_{iij})^2\\
&+\bigl\{\dfrac{13}{3c_1}(\dfrac{10}{13}+c_1)^2-t\bigl\}  \biggl[\frac 13(A+2B)
-\frac {4}3\sum_j\frac 1{S+2\lambda_j^2}\big( \sum_i
\lambda_i^2h_{iij}\bigl)^2\biggl] \\
&=-4A+5B+\sum_{j}\bigl(\dfrac{13(1+c_1)}S
-\dfrac{4\lambda_j^2}{tS^2}-\dfrac {56}{3(S+2\lambda_j^2)}\bigl)\big( \sum_i
\lambda_i^2h_{iij}\bigl)^2,\\
\end{aligned}
\end{equation*}
where in the last equality, we used (\ref{eq:3.8}). 
Because of  
$$
-\dfrac{4\lambda_j^2}{tS^2}-\dfrac {56}{3(S+2\lambda_j^2)}=\dfrac2{tS}-\dfrac{2S+4\lambda_j^2}{tS^2}-\dfrac {56}{3(S+2\lambda_j^2)},
$$
by making use of $x^2+y^2\geq 2xy$ for the last two terms in the above equality, we have
\begin{equation}\label{eq:3.9}
\begin{aligned}
\phi&\leq -4A+5B+\dfrac{\eta(t)}S\sum_{j}\big( \sum_i\lambda_i^2h_{iij}\bigl)^2.\\
\end{aligned}
\end{equation}
 where $\eta(t)$ is defined by
 $$
 \begin{aligned}
 \eta(t)&= 13(1+c_1)+\dfrac{2}{t}-8\sqrt{\dfrac{7}{3t}}\\
&=\bigl(24+\dfrac{3t-\sqrt{(22+3t)^2-400}}{2}+\dfrac{2}{t}
-\dfrac{8\sqrt {21}}{3}\dfrac1{\sqrt{t}}\bigl).
\end{aligned}
 $$
 Since $\dfrac{d\eta(t)}{dt}>0$ for  $\dfrac16\leq t \leq \dfrac 27$,  $\eta(t)$ 
 is an increasing function when $\dfrac16\leq t \leq \dfrac 27$. Thus,
 we have $1.1628<\eta(1/6)\leq \eta(t)\leq \eta(2/7)\simeq 3.033833811$.
According to the inequality for $\big( \sum_i\lambda_i^2h_{iij}\bigl)^2$ in  the lemma 3.1 and (\ref{eq:3.9}), we obtain
\begin{equation}\label{eq:3.10}
\begin{aligned}
\phi&\leq -4A+5B+t\eta(t) \frac{1+2\alpha}3Sf.\\
\end{aligned}
\end{equation}
From the definition of $\phi$ and (\ref{eq:3.10}), we obtain
\begin{equation*}
\begin{aligned}
&-(\dfrac{13+\eta(t)}{3}+ \dfrac{2\eta(t)\alpha}3)tSf-\frac {2f_3}S \sum_{i,j,k}\lambda_ih_{ijk}^2  
-\frac 1{S^2}\bigl(\sum_{i,j,k}\lambda_ih_{ijk}^2 \bigl)^2
\\&-\dfrac{4}{tS^2}\sum_{j}\lambda_j^2(\sum_i\lambda_i^2h_{iij})^2\leq -4A+5B.\\
\end{aligned}
\end{equation*}
Letting $a(t)=\dfrac{13+\eta(t)}{3}>\dfrac{14}3$ and $b(t)=\dfrac{2\eta(t)}3>\dfrac23$ for $\dfrac16\leq t \leq \dfrac 27$, we know 
that $a(t)$ and $b(t)$ are the increasing functions  for $\dfrac16\leq t \leq \dfrac 27$. Hence, we have
\begin{equation}\label{eq:3.11}
\begin{aligned}
&-(a+ b\alpha)tSf-\frac {2f_3}S \sum_{i,j,k}\lambda_ih_{ijk}^2  \\
&-\frac 1{S^2}\bigl(\sum_{i,j,k}\lambda_ih_{ijk}^2 \bigl)^2
-\dfrac{4}{tS^2}\sum_{j}\lambda_j^2(\sum_i\lambda_i^2h_{iij})^2\\
&\leq -(a(t)+ b(t)\alpha)tSf-\frac {2f_3}S \sum_{i,j,k}\lambda_ih_{ijk}^2  \\
&-\frac 1{S^2}\bigl(\sum_{i,j,k}\lambda_ih_{ijk}^2 \bigl)^2
-\dfrac{4}{tS^2}\sum_{j}\lambda_j^2(\sum_i\lambda_i^2h_{iij})^2\\
&\leq -4A+5B,\\
\end{aligned}
\end{equation}
where $a=a(\frac27)$ and $b=b(\frac27)$.
In order to prove our assertion from (\ref{eq:3.11}), we divide it into two steps.
\vskip1mm
\noindent
Step (1). If 
$$
-\frac {2f_3}S \sum_{i,j,k}\lambda_ih_{ijk}^2  
-\frac 1{S^2}\bigl(\sum_{i,j,k}\lambda_ih_{ijk}^2 \bigl)^2
-\dfrac{4}{tS^2}\sum_{j}\lambda_j^2(\sum_i\lambda_i^2h_{iij})^2> 0,
$$
from (\ref{eq:3.11}), we obtain
\begin{equation}\label{eq:3.12}
\begin{aligned}
&-0.92atSf\leq \dfrac{0.92a}{a+ b\alpha}(-4A+5B),\\
\end{aligned}
\end{equation}
because of  $1.1628<\eta(1/6)$.
From (\ref{eq:3.2}), (\ref{eq:3.3})  and the above inequality, we obtain
\begin{equation*}
\begin{aligned}
&(\dfrac32-0.92at)Sf\leq \dfrac{0.92a}{a+ b\alpha}(-4A+5B)+tS^2(S-2)+3(A-2B)\\
&=-\dfrac{0.92a}{a+ b\alpha}\big\{\dfrac{13}{3}(A-B)-\dfrac13(A+2B)\bigl\}\\
&+4(A-B)-(A+2B)+tS^2(S-2),\\
\end{aligned}
\end{equation*}
that is,
\begin{equation*}
\begin{aligned}
&(\dfrac32-0.92at)Sf
\leq (4-\dfrac{0.92a}{a+ b\alpha}\dfrac{13}{3})(A-B)+tS^2(S-2)\\
\end{aligned}
\end{equation*}
according to  $A+2B\geq 0$.
From $\eta(2/7)\simeq3.033833811$, we have 
$$
a=\dfrac{13+\eta(2/7)}{3}\simeq5.34461127, \ \ b=\dfrac{2\eta(2/7)}3\simeq2.02255587, \ \  w=\dfrac{b}{a}\simeq0.378429
$$
and 
$$
\dfrac32-0.92at\geq \dfrac32-0.92a\dfrac27>0, \ \ 4-\dfrac{0.92a}{a+ b\alpha}\dfrac{13}{3}\geq 4-\dfrac{0.92\times13}{3}>0.
$$
These inequalities in  lemma 3.1 for $Sf$ and $A-B$ and the above inequality yield
\begin{equation}\label{eq:3.13}
\begin{aligned}
&(\dfrac3{2t}-0.92a)\dfrac{(\lambda_1-\lambda_2)^2}{\lambda_1^2+\lambda_2^2}(\lambda_1\lambda_2)^2\\
&\leq (4-\dfrac{0.92a}{a+ b\alpha}\dfrac{13}{3})\dfrac {1-\alpha}3(\lambda_1-\lambda_2)^2S+S(S-2).\\
\end{aligned}
\end{equation}
We infer
 \begin{equation*}
\begin{aligned}
&0\leq \bigl(12-\dfrac{0.92\times 13}{1+w\alpha}\bigl)(1-\alpha)\\
&=\bigl(12-\dfrac{0.92\times13}{1+ w\alpha}\bigl)\bigl(1+w-(1+ w\alpha)\bigl)\dfrac{1}{w}\\
&=\dfrac1w\bigl(12(1+w)+0.92\times13-12(1+w\alpha)-\dfrac{0.92\times13(1+w)}{1+ w\alpha}\bigl)\\
&\leq \dfrac1w\bigl(12(1+w)+0.92\times13-4\sqrt{0.92\times39(1+w)}\bigl)\\
&=\dfrac1w(2\sqrt{3(1+w)}-\sqrt{0.92\times13})^2,
\end{aligned}
\end{equation*}
where we used $x^2+y^2\geq 2xy$ in the last inequality.
Hence, we get from (\ref{eq:3.13})
\begin{equation}\label{eq:3.14}
\begin{aligned}
&(\dfrac3{2t}-0.92\times a)\dfrac{(\lambda_1-\lambda_2)^2}{\lambda_1^2+\lambda_2^2}(\lambda_1\lambda_2)^2\\
&\leq\dfrac1{9w}(2\sqrt{3(1+w})-\sqrt{0.92\times13})^2(\lambda_1-\lambda_2)^2S+S(S-2).
\end{aligned}
\end{equation}
From (\ref{eq:3.14}), we obtain
\begin{equation}\label{eq:3.15}
\begin{aligned}
&0.33295\dfrac{(\lambda_1-\lambda_2)^2}{\lambda_1^2+\lambda_2^2}(\lambda_1\lambda_2)^2
\leq0.10881(\lambda_1-\lambda_2)^2S+S(S-2).
\end{aligned}
\end{equation}
From (\ref{eq:3.15}) and $(\lambda_1-\lambda_2)^2\leq 2S$, we have
\begin{equation}\label{eq:3.16}
\begin{aligned}
&0
\leq 0.21762S+(S-2).
\end{aligned}
\end{equation}
We infer  
$$
S> 1.64>\dfrac75.
$$
\vskip1mm
\noindent
Step (2). If 
$$
-\frac {2f_3}S \sum_{i,j,k}\lambda_ih_{ijk}^2  
-\frac 1{S^2}\bigl(\sum_{i,j,k}\lambda_ih_{ijk}^2 \bigl)^2
-\dfrac{4}{tS^2}\sum_{j}\lambda_j^2(\sum_i\lambda_i^2h_{iij})^2\leq 0,
$$
because of  $1.1628<\eta(1/6)$, we infer from (\ref{eq:3.11})
\begin{equation}\label{eq:3.18}
\begin{aligned}
&-atSf-\frac {2f_3}S \sum_{i,j,k}\lambda_ih_{ijk}^2  -\frac 1{S^2}\bigl(\sum_{i,j,k}\lambda_ih_{ijk}^2 \bigl)^2-\dfrac{4}{tS^2}\sum_{j}\lambda_j^2(\sum_i\lambda_i^2h_{iij})^2\\
&\leq \dfrac{a}{a+ b\alpha}(-4A+5B).\\
\end{aligned}
\end{equation}
Adding  (\ref{eq:3.4}) in the proposition 3.1 and  (\ref{eq:3.18}), we conclude
 \begin{equation*}
\begin{aligned}
&(2-at)Sf\\
&\leq  \dfrac{a}{a+ b\alpha}(-4A+5B)+S(S-1)(S-2) +5A-6B\\
&=\dfrac13\bigl(16-\dfrac{13a}{a+ b\alpha}\bigl)(A-B)+tS^2(S-2) 
-\dfrac13(1-\dfrac{a}{a+ b\alpha})(A+2B)\\
&\leq\dfrac13\bigl(16-\dfrac{13a}{a+ b\alpha}\bigl)(A-B)
+tS^2(S-2).
\end{aligned}
\end{equation*}
In the second equality and the last inequality, we used 
$$
-4A+5B=-\dfrac{13}3(A-B)+\dfrac13(A+2B), \quad 
5A-6B=\dfrac{16}{3}(A-B)-\dfrac{1}3(A+2B)
$$
and  $A+2B\geq 0$, respectively.
Thus, we have 
 \begin{equation}\label{eq:3.19}
\begin{aligned}
(2-at)Sf
&\leq\dfrac13\bigl(16-\dfrac{13a}{a+ b\alpha}\bigl)(A-B)
+tS^2(S-2).
\end{aligned}
\end{equation}
From the inequalitiy in  lemma 3.1 for  $A-B$, we obtain
 \begin{equation}\label{eq:3.20}
\begin{aligned}
&(2-at)Sf\\
&\leq\dfrac13\bigl(16-\dfrac{13a}{a+ b\alpha}\bigl)\dfrac {1-\alpha}3(\lambda_1-\lambda_2)^2tS^2
+tS^2(S-2).
\end{aligned}
\end{equation}
From $w=\dfrac{b}{a}$, we have, by using $x^2+y^2\geq 2xy$, 
 \begin{equation}\label{eq:3.21}
\begin{aligned}
&\bigl(16-\dfrac{13}{1+w\alpha}\bigl)(1-\alpha)\\
&=\dfrac1w\bigl(16(1+w)+13-16(1+w\alpha)-\dfrac{13(1+w)}{1+ w\alpha}\bigl)\\
&\leq \dfrac1w\bigl(16(1+w)+13-8\sqrt{13(1+w)}\bigl)=\dfrac1w(4\sqrt{1+w}-\sqrt{13})^2.
\end{aligned}
\end{equation}
In view of 
 $$
 f=f_4-\frac{f_3^2}{S}\geq \dfrac{(\lambda_1-\lambda_2)^2}{\lambda_1^2+\lambda_2^2}(\lambda_1\lambda_2)^2,
$$
in the lemma 3.1,  (\ref{eq:3.20}) and (\ref{eq:3.21}),  we obtain
\begin{equation}\label{eq:3.22}
\begin{aligned}
&(\dfrac2t-a)\dfrac{(\lambda_1-\lambda_2)^2}{\lambda_1^2+\lambda_2^2}(\lambda_1\lambda_2)^2\\
&\leq\dfrac1{9w}(4\sqrt{1+w}-\sqrt{13})^2(\lambda_1-\lambda_2)^2S+S(S-2).
\end{aligned}
\end{equation}
From $\eta(2/7)\simeq3.033833811$, we have 
$$
a=\dfrac{13+\eta(2/7)}{3}\simeq5.34461127, \ \ b=\dfrac{2\eta(2/7)}3\simeq2.02255587, \ \  w\simeq0.378429.
$$
If $\lambda_1\lambda_2\leq 0$, we get from (\ref{eq:3.22})
\begin{equation*}
\begin{aligned}
&(\dfrac2t-a)(\lambda_1\lambda_2)^2
+\dfrac2{9w}(4\sqrt{1+w}-\sqrt{13})^2S\lambda_1\lambda_2\leq \dfrac1{9w}(4\sqrt{1+w}-\sqrt{13})^2S^2+S(S-2).
\end{aligned}
\end{equation*}
Thus, we have 
\begin{equation*}
\begin{aligned}
&-\dfrac{\bigl\{\dfrac1{9w}(4\sqrt{1+w}-\sqrt{13})^2\bigl\}^2}{(\dfrac2t-a)}S
\leq \dfrac1{9w}(4\sqrt{1+w}-\sqrt{13})^2S+(S-2).
\end{aligned}
\end{equation*}
We  used $(x+y)^2\geq 0$ in the above inequality.
Therefore, by using the value of $w$, we conclude
\begin{equation*}
\begin{aligned}
&-0.07377S
\leq 0.3492941S+(S-2),
\end{aligned}
\end{equation*}
that is, $S\geq 1.4054>1+\dfrac25$.\\
If $\lambda_1\lambda_2>0 $, we get from (\ref{eq:3.22})
\begin{equation*}
\begin{aligned}
&0\leq \dfrac1{9w}(4\sqrt{1+w}-\sqrt{13})^2S^2+S(S-2).
\end{aligned}
\end{equation*}
Hence, we have
\begin{equation*}
\begin{aligned}
&0\leq 0.3492941S+(S-2).\\
\end{aligned}
\end{equation*}
Hence $S\geq 1.4822>1+\dfrac25$.
 We complete the proof of the theorem 3.1.
 \begin{flushright}
 $\square$ 
\end{flushright}

\vskip2mm
\noindent
{\it Proof of theorem 1.2}. Since $S$ is constant,  according to the above theorem 3.1,  if  $S\leq \dfrac75$ we have $S\leq 1$ and 
the self-shrinker  $X: M\rightarrow \mathbb{R}^{n+1}$
is isometric to one of  $S^{n}(\sqrt{n})$,  $\mathbb{R}^{n}$, 
$S^k (\sqrt k)\times \mathbb{R}^{n-k}$, $1\leq k\leq n-1$.
\newline
Next, we will prove  that $S$ satisfies $S\leq \dfrac75$. If  $S>\dfrac75$, 
since $S$ is constant, for this fixed self-shrinker, $S>\dfrac75+\epsilon$ for a very small
$\epsilon>0$ depending on this self-shrinker.
According to the Gauss equation (\ref{eq:2.2}), we infer that the Ricci curvature is bounded from below
and $H^2=S+R>\epsilon$ since $R\geq -\frac 75$. 
Thus, we  can assume $H>\sqrt{\epsilon}$ on $M$. From (\ref{eq:2.9}), we have
\begin{equation}\label{eq:3.23}
\begin{aligned}
\dfrac12\mathcal{L}\dfrac {S} {H^2}
&=\dfrac{1}{H^4}\sum_{i, j, k}|h_{ij}\nabla_kH-h_{ijk}H|^2-\dfrac1H\langle\nabla H,\nabla \dfrac {S}{H^2}\rangle.
\end{aligned}
\end{equation}
Applying the generalized maximum principle for $\mathcal L$-operator to the function $\dfrac{S}{H^2}$, 
there exists a sequence $\{p_m\} \subset M^n$ such that 
\begin{enumerate}
\item $\lim_{m\to\infty}\dfrac{S}{H^2}(p_m)=\sup \dfrac{S}{H^2}$,
\item $\lim_{m\to\infty}|\nabla \dfrac{S}{H^2}|(p_m)=0$,
\item $\lim\sup_{m\to\infty} \mathcal L\dfrac{S}{H^2}(p_m)\leq 0$.
\end{enumerate}
Since $S$ is constant, we obtain  
$$
|\nabla \dfrac{S}{H^2}|^2=\dfrac{4S^2}{H^6}|\nabla H|^2.
$$
We derive 
\begin{equation}\label{eq:3.24}
\lim_{m\to\infty}|\nabla H(p_m)=0.
\end{equation}
Hence, we get, from (\ref{eq:3.23}) and (\ref{eq:3.24}),
$\lim_{m\to\infty}h_{ijk}(p_m)=0$ for any $i, \ j, \ k$.
In view to (\ref{eq:2.7}) and $S=$constant, we have
$\lim_{m\to\infty}{S}(p_m)=1$ or  $\lim_{m\to\infty}{S}(p_m)=0$.
This is impossible because of $S>\dfrac{7}{5}$.
We complete our proof of the theorem 1.2.
\begin{flushright}
 $\square$ 
\end{flushright}

\section{Bounds of scalar curvature}

\noindent
In this section, we will discuss  upper bounds of the scalar curvature and  prove the theorem 1.3.
\vskip2mm
\noindent
{\it Proof of theorem 1.3}.
Since the  scalar curvature  is a positive constant, we know  
\begin{equation*}
0<R=H^2-S\leq H^2-\frac{H^2}n=\frac{n-1}nH^2,
\end{equation*}
that is,  
$$
H^2\geq \frac{n}{n-1}R>0.
$$
We can assume $H>0$.
From the Gauss equation (\ref{eq:2.2}), we obtain 
\begin{equation}
1>1-\dfrac{R}{H^2}=\dfrac{S}{H^2}\geq \dfrac{1}{n}.
\end{equation}
Since the scalar curvature is constant, we have 
$$
\nabla_iH^2=\nabla_iS
$$
for any $i$, 
that is, 
$$
H\nabla_i H=\sum_{j}\lambda_jh_{jji},
$$
where $\lambda_j$'s  denote the principal curvatures. 
Therefore, we obtain from the Schwarz inequality
\begin{equation}
H^2|\nabla H|^2\leq S\sum_{i,j}h_{iij}^2.
\end{equation}
From $\dfrac{S}{H^2}<1$, we have 
$$
|\nabla H|^2\leq \sum_{i,j}h_{iij}^2.
$$
In view of 
\begin{equation}\label{eq:4.3}
0=\dfrac12\mathcal L R=\dfrac{1}{2}\mathcal{L}(H^2-S)
=|\nabla H|^2-\sum_{i,j,k}h_{ijk}^{2}+(1-S)R,
\end{equation}
we have $S\leq 1$. Hence $H^2\leq nS\leq n$. Thus, we conclude 
$$
\dfrac{n}{n-1}R\leq H^2\leq n, 
$$
which yields 
$$
 0<R\leq n-1.
$$
\begin{flushright}
 $\square$ 
\end{flushright}

\vskip5mm
\section{Proofs of theorems 1.4 and 1.5}
\vskip2mm
\noindent
In this section, we will prove the theorems 1.4 and 1.5.
\vskip2mm
\noindent
{\it Proof of theorem 1.4}.  If the scalar curvature $R=0$, the result has been proven by Luo, Sun and Yin \cite{lsy}. Hence, we only
consider  the case that  the  scalar curvature is a positive constant.  According to the theorem 1.3 and its proof,  we know 
 $$
 0<R\leq n-1,  \quad H^2\geq \frac{n}{n-1}R>0.
 $$
We can assume $H>0$.
Hence, we obtain
\begin{equation}\label{eq:5.1}
\nabla_i\dfrac{S}{H^2}=\dfrac{R}{H^4}\nabla_iH^2, \quad |\nabla\dfrac{S}{H^2}|^2=\dfrac{4R^2}{H^6}|\nabla H|^2.
\end{equation}
From  the formula (\ref{eq:2.9}), we have
\begin{equation}\label{eq:5.2} 
\begin{aligned}
\dfrac12\mathcal{L}\dfrac {S} {H^2}
&=\dfrac{1}{H^4}\sum_{i, j, k}|h_{ij}\nabla_kH-h_{ijk}H|^2-\dfrac1H\langle\nabla H,\nabla \dfrac {S}{H^2}\rangle.
\end{aligned}
\end{equation}
In view of the  theorem 1.3, we have  $S\leq 1$. Hence,  from the Gauss equation (\ref{eq:2.2}), 
we know that the Ricci curvature is bounded from below.  Applying
the generalized maximum principle due to Cheng and Peng \cite{cp} to the function $\dfrac{S}{H^2}$, 
there exists a sequence $\{p_k\} \subset M^n$ such that 
\begin{equation*}
\lim_{k\to\infty}\dfrac{S}{H^2}(p_k)=\sup \dfrac{S}{H^2},  \  \ \lim_{k\to\infty}|\nabla \dfrac{S}{H^2}|(p_k)=0, \  \
 \lim\sup_{k\to\infty} \mathcal L\dfrac{S}{H^2}(p_k)\leq 0.
\end{equation*}
In view of (\ref{eq:5.1}), we obtain
$$
\lim_{k\to\infty}\dfrac1H\langle\nabla H,\nabla \dfrac {S}{H^2}\rangle(p_k)=\lim_{k\to\infty}\dfrac{2R}{H^4}|\nabla H|^2(p_k)=0.
$$
From (\ref{eq:5.2}), $ \lim\sup_{k\to\infty} \mathcal L\dfrac{S}{H^2}(p_k)\leq 0$  and the above equality, we derive
$$
\lim_{m\to\infty}\dfrac{1}{H^4}\sum_{i, j, k}|h_{ij}\nabla_kH-h_{ijk}H|^2(p_m)= 0.
$$
Hence, we know  for any $i, j, k$,
$$
\lim_{m\to\infty}|\nabla H|(p_m)=0, \quad \lim_{m\to\infty}h_{ijk}(p_m)=0.
$$
Since the scalar curvature is constant, from the Gauss equation (\ref{eq:2.2}) and (\ref{eq:2.7}), we have
\begin{equation}\label{eq:}
0=\dfrac12\mathcal L R=\dfrac{1}{2}\mathcal{L}(H^2-S)
=|\nabla H|^2-\sum_{i,j,k}h_{ijk}^{2}+(1-S)R.
\end{equation}
Thus, we get 
\begin{equation}
\lim_{k\to\infty}S(p_k)=1=\sup S
\end{equation}
from $S\leq 1$.
Because of 
\begin{equation}
\sup \dfrac{S}{H^2}=\lim_{k\to\infty}\dfrac{S}{H^2}(p_k)=\dfrac{1}{R+1},
\end{equation}
we have 
$$
\dfrac{S}{H^2}\leq \dfrac{1}{R+1} \ {\rm and}\ S\leq \dfrac{H^2}{R+1}.
$$
If $S(p)=1$ for some $p\in M$, letting $u=1-S\geq 0$, from (\ref{eq:2.7}), we have
\begin{equation}
\dfrac{1}{2}\mathcal{L}u
=-\sum_{i,j,k}h_{ijk}^{2}-uS\leq 0.
\end{equation}
According to the strong maximum principle,
we know $S\equiv 1$. 
Hence,  $X : M^{n}\to \mathbb{R}^{n+1}$ is isometric to  $S^{n}(\sqrt{n})$
or $S^k (\sqrt k)\times \mathbb{R}^{n-k}$, $1\leq k\leq n-1$.\\
Otherwise, we have $S<1$ and $\sup S=1$. 
From $H^2\geq \frac{n}{n-1}R>0$  and $H^2\leq nS<n$, we obtain
$0<R<n-1$.
The Gauss  equation $H^2-S=R$ implies 
$$
1>S\geq \dfrac{R}{n-1}\ \ {\rm and} \ \ \sup S+R =\sup H^2\leq n
$$
since $R$ is constant.  For any $j$,  from 
\begin{equation*}
\begin{aligned}
H^2-R&=S=\sum_{i=1}^n\lambda_i^2\geq \dfrac{1}{n-1}(\lambda_j-H)^2+\lambda_j^2,
\end{aligned}
\end{equation*}
we obtain
\begin{equation*}
\begin{aligned}
0\geq n\lambda_j^2-2H\lambda_j-(n-2)H^2+(n-1)R.
\end{aligned}
\end{equation*}
Thus, we derive
\begin{equation}\label{eq:5.7}
\begin{aligned}
&\dfrac{1}n\bigl(H-\sqrt{(n-1)^2H^2-n(n-1)R}\bigl)\\
&\leq\lambda_j\leq \dfrac{1}n\bigl(H+\sqrt{(n-1)^2H^2-n(n-1)R}\bigl).
\end{aligned}
\end{equation}
Hence, we have
\begin{equation*}
\begin{aligned}
&-\dfrac{1}n\bigl(\dfrac{n-2}2H+\sqrt{(n-1)^2H^2-n(n-1)R}\bigl)\\
&\leq\lambda_j-\dfrac H2\leq \dfrac{1}n\bigl(-\dfrac{n-2}2H+\sqrt{(n-1)^2H^2-n(n-1)R}\bigl).
\end{aligned}
\end{equation*}
According to the Gauss equation (\ref{eq:2.2}), we get
\begin{equation*}
\begin{aligned}
&R_{jj}=H\lambda_j-\lambda_j^2=\dfrac{H^2}4-(\lambda_j-\dfrac{H}2)^2\\
&\geq \dfrac{H^2}4-\dfrac{1}{n^2}\bigl(\dfrac{n-2}2H+\sqrt{(n-1)^2H^2-n(n-1)R}\bigl)^2\\
&=\dfrac{1}{n^2}\bigl((n-1)H+\sqrt{(n-1)^2H^2-n(n-1)R}\bigl)\bigl((H-\sqrt{(n-1)^2H^2-n(n-1)R}\bigl)\\
&=\dfrac{1}{n^2}\dfrac{(n-1)H+\sqrt{(n-1)^2H^2-n(n-1)R}}{(H+\sqrt{(n-1)^2H^2-n(n-1)R}}\bigl(n(n-1)R-n(n-2)H^2\bigl).\\
\end{aligned}
\end{equation*}
If $R>(n-2)$, we have
$$
n(n-1)R-n(n-2)H^2=nR-n(n-2)S>n(R-(n-2))>0.
$$
We know 
$$
R_{jj}>\dfrac{R-(n-2)}{n}>0.
$$
In view of the Myers theorem,  $X : M^{n}\to \mathbb{R}^{n+1}$  is compact. It is a contradiction.
Therefore, $R\leq n-2$. From the Gauss equation (\ref{eq:2.2}), 
we  conclude 
$$ 
H^2= S+R< n-1.
$$
If $R=n-2$ and  there exists a point $p$ such that some  principal curvature $\lambda_j$ at $p$
is non-positive, in view of (\ref{eq:5.7}),  we have  at $p$, 
$$
H-\sqrt{(n-1)^2H^2-n(n-1)R}\leq 0.
$$
Hence, 
$$
H^2\geq \dfrac{n-1}{n-2}R=n-1.
$$
 It is impossible because of $H^2<n-1$.
Thus,  all of the  principal curvatures are positive on $M$. From the theorem of Stoker \cite{jn,s, v},  we know that
$X : M^{n}\to \mathbb{R}^{n+1}$ is the  boundary of a convex body in $\mathbb{R}^{n+1}$ 
and diffeomorphic to $\mathbf R^n$.
Hence,  $X : M^{n}\to \mathbb{R}^{n+1}$ is proper.  According to the theorem of Colding and Minicozzi \cite{cm}, 
it is also impossible because of $S<1$. We should remark that the embedding condition in the theorem of 
Colding and Minicozzi \cite{cm} is not necessary if the  scalar curvature is positive.
Hence, $R<n-2$ and $H^2<R+1<n-1$.\\
From 
$$
\nabla_i|X|^2=2\langle X,e_i\rangle, 
$$
we have 
\begin{equation}\label{eq:5.8}
|\nabla |X|^2|^2=4\sum_i\langle X,e_i\rangle^2.
\end{equation}
Applying the generalized maximum principle for $\mathcal L$-operator  to $-|X|^2$,  we know
there exists a sequence $\{p_k\} \subset M^n$ such that 
$$
\lim_{k\to\infty}|X|^2(p_k)=\inf |X|^2,
\ \ \lim_{k\to\infty}|\nabla |X|^2|(p_k)=0,
\ \  \lim\inf_{k\to\infty} \mathcal L|X|^2(p_k)\geq 0.
$$
In view of (\ref{eq:2.6}) and (\ref{eq:5.8}), we derive
$$
\lim_{k\to\infty}|\nabla |X|^2|^2(p_k)=4\lim_{k\to\infty}\sum_i\langle X,e_i\rangle^2(p_k)=0
$$
and
$$
\inf |X|^2= \lim_{k\to\infty}|X|^2(p_k)=\lim_{k\to\infty}H^2(p_k)\geq \frac{n}{n-1}R,
$$
that is, 
$$
 \inf |X|^2\geq \frac{n}{n-1}R.
$$
Furthermore, if $\sup |X|^2<\infty$, by applying the generalized maximum principle for $\mathcal L$-operator  to $|X|^2$,
we obtain
 $$
 \lim_{k\to\infty}|X|^2(p_k)=\sup|X|^2,
\ \ \lim_{k\to\infty}|\nabla |X|^2|(p_k)=0,
$$
$$
  \lim\sup_{k\to\infty} \dfrac12\mathcal L|X|^2(p_k)=n-\sup |X|^2\leq 0.
$$
According to (\ref{eq:2.6}) and (\ref{eq:5.8}), we have
$$
\sup |X|^2= \lim_{k\to\infty}|X|^2(p_k)=\lim_{k\to\infty}H^2(p_k)\geq n.
$$
This is a contradiction because of $H^2<n-1$. Hence, we conclude  $\sup |X|^2=\infty$.
We finish our proof of the theorem 1.4. 
\begin{flushright}
 $\square$ 
\end{flushright}
\vskip2mm
\noindent
{\it Proof of theorem 1.5}. If $R=0$, then $H\equiv 0$ and $S\equiv 0$. Hence, $X : M^{n}\to \mathbb{R}^{n+1}$ is totally 
geodesic. Thus, $X : M^{n}\to \mathbb{R}^{n+1}$ is isometric to $\mathbb R^n$.
Next, we consider $R>0$.  According to the theorem 1.4, we have that 
$X : M^{n}\to \mathbb{R}^{n+1}$ is isometric to $S^{n}(\sqrt{n})$ in this case $H^2=\dfrac{n}{n-1}R$ 
or  $S^{n-1}(\sqrt {n-1})\times \mathbb{R}$ in this case  $H^2=\dfrac{n-1}{n-2}R$ or $X : M^{n}\to \mathbb{R}^{n+1}$
satisfies $R< n-2$, $\dfrac{n}{n-1}R\leq H^2$, $S<1$ and $\sup S=1$. From (\ref{eq:5.7}), for any $j$, we have
\begin{equation*}
\begin{aligned}
&\dfrac{1}n\bigl(H-\sqrt{(n-1)^2H^2-n(n-1)R}\bigl)\leq\lambda_j\leq \dfrac{1}n\bigl(H+\sqrt{(n-1)^2H^2-n(n-1)R}\bigl).
\end{aligned}
\end{equation*}
If there exists a point $p\in M$  such that  for some $j$, $\lambda_j\leq 0$ at $p$,  we have from (\ref{eq:5.7})
$$
H-\sqrt{(n-1)^2H^2-n(n-1)R}\leq0.
$$
Hence, we get 
$$
(n-1)R\leq (n-2)H^2.
$$
Since $H^2\leq \dfrac{n-1}{n-2}R$ holds, we know that $H^2$ attains its maximum at $p$,
$$
H^2(p)=\sup H^2=\dfrac{n-1}{n-2}R. 
$$
From $S<1$ on $M$ and $\sup S=1$, we know that $S$ does not  get its maximum on $M$. Therefore,
from the Gauss equation (\ref{eq:2.2}) $H^2=S+R$, we conclude that $H^2$ can not attain its maximum
on $M$. This is a contradiction. Hence, we conclude that all of the  principal curvatures are positive on $M$. 
$X : M^{n}\to \mathbb{R}^{n+1}$ is locally, strictly  convex. 
From the theorem of Stoker \cite{jn,s, v},  we know that
$X : M^{n}\to \mathbb{R}^{n+1}$ is the  boundary of a convex body in $\mathbb{R}^{n+1}$ 
and diffeomorphic to $\mathbf R^n$.
Hence,  $X : M^{n}\to \mathbb{R}^{n+1}$ is proper.  According to the theorem of Colding and Minicozzi \cite{cm}, 
it is also impossible because of $S<1$. Thus, 
$X : M^{n}\to \mathbb{R}^{n+1}$ is isometric to $S^{n}(\sqrt{n})$
or  $S^{n-1}(\sqrt {n-1})\times \mathbb{R}$.

\begin{flushright}
 $\square$ 
\end{flushright}
\vskip2mm
\noindent

 \vskip2mm
\noindent
{\bf Acknowledgement}.
Authors would like to express their sincere gratitude to the  referee for valuable suggestions and comments. \newline  
This work was partly supported by MEXT Promotion of Distinctive Joint Research Center Program JPMXP0723833165 
and Osaka Metropolitan University Strategic Research Promotion Project (Development of International Research Hubs).


\begin{thebibliography}{99}
\bibitem{al}
U. Abresch and J. Langer, The normalized curve shortening flow and
homothetic  solutions, J. Differential Geom., {\bf 23}(1986), 175-196.

\bibitem{b}
S. Brendle,  Embedded self-similar shrinkers of genus 0,   Ann. of Math.,  {\bf 183}(2016), 715-728.

\bibitem{cl}
H.-D. Cao and H. Li,  A gap theorem for self-shrinkers of the mean
curvature flow in arbitrary codimension, Calc. Var. Partial Differential Equations, {\bf 46} (2013), 879-889.


\bibitem{cliw} Q. -M. Cheng, Z. Li and G. Wei, Complete self-shrinkers with constant norm of the second fundamental form,  
{Math. Z.},  300 (2022), 995-1018.


\bibitem{cliw1} Q. -M. Cheng, Z. Li and G. Wei,   A classification of complete $3$-dimensional self-shrinkers in the Euclidean space $\mathbb R^{4}$,
Sci. China Math., 67(2024) , 873-882. 


\bibitem{co}
Q. -M. Cheng and S. Ogata, $2$-dimensional complete  self-shrinkers in $\mathbb R^{3}$,
  Math. Z., {\bf 284}(2016), 537-542.

\bibitem{cp}
Q. -M. Cheng and Y. Peng,  Complete  self-shrinkers of the mean curvature flow,
  Calc. Var. Partial Differential Equations, {\bf 52} (2015), 497-506.

\bibitem{cw}
Q. -M. Cheng and G. Wei,  A gap theorem for self-shrinkers,
Trans. Amer. Math. Soc., {\bf 367} (2015), 4895-4915.


\bibitem{cwy} Q. -M. Cheng,  G. Wei and W. Yano, The second gap on complete self-shrinkers, 
{Proc. Amer. Math. Soc.}  {\bf 151} (2023), 339-348. doi.org/10.1090/proc/16107.	

\bibitem{cz}
X. Cheng and D. Zhou,  Volume estimate about shrinkers, Proc. Amer. Math. Soc., {\bf 141} (2013), 687-696.

\bibitem{cm}
T. H. Colding and W. P.  Minicozzi II,  Generic mean curvature flow I;  Generic singularities,
Ann. of Math.,  {\bf 175} (2012), 755-833.

\bibitem{dx1}
Q. Ding and Y. L. Xin, Volume growth, eigenvalue and compactness for self-shrinkers,  Asian J. Math., {\bf 17} (2013), 443-456.

\bibitem{dx2}
Q. Ding and Y. L. Xin, The rigidity theorems of self shrinkers, Trans. Amer. Math. Soc., {\bf 366} (2014), 5067-5085.

\bibitem{dr}
G. Drugan, An immersed $S^2$ self-shrinker, Trans. Amer. Math. Soc. {\bf 367} (2015), 3139-3159.

\bibitem{g}
Z. Guo, Scalar curvature of self-shrinkers, J. Math. Soc. Japan, {\bf 70} (2018), 1103-1110.


\bibitem{h}
H. Halldorsson, Self-similar solutions to the curve shortening flow, Trans. Amer. Math. Soc., {\bf 364} (2012), 5285-5309.

\bibitem{h2} G. Huisken,
Asymptotic behavior for singularities of the mean curvature flow, J. Differential
Geom., {\bf 31} (1990), 285-299.

\bibitem{h3}
G. Huisken, Local and global behaviour of hypersurfaces moving by mean curvature,  Differential
geometry: partial differential equations on manifolds (Los Angeles, CA, 1990), Proc. Sympos. Pure
Math., {\bf 54}, Part 1, Amer. Math. Soc., Providence, RI, (1993), 175-191.

\bibitem{jn} 
L. B. Jonker and  R. D. Norman, {\it Locally convex hypersurfaces}, Can. J. Math., {\bf 25} (1973), 531-538.

\bibitem {l}  Lawson H. B. Jr., Local rigidity theorems for
minimal hypersurfaces,  Ann. of  Math., \textbf{89} (1969),
187-197.

\bibitem{lw1}
H. Li and Y.  Wei, {\it Lower volume growth estimates for self-shrinkers of mean curvature flow}, 
Proc. Amer. Math. Soc.,  {\bf 142} (2014), 3237-3248.

\bibitem{lw2}
H. Li and Y. Wei, {\it Classification and rigidity of self-shrinkers in the mean curvature flow}, J. Math. Soc. Japan, {\bf 66} (2014), 709-734.

\bibitem{lsy}
Y. Luo, L. Sun and J. Yin, {\it Complete self-similar hypersurfaces to mean curvature flow with non-negative constant scalar curvature}, 
Front. Math., {\bf 18} (2023), 417-430.


\bibitem{s}	 J. Stoker, {\it ber die gestalt der positiv gekmmten  offenen  flchen  im  dreidimensionalen  raume}, 
Compositio Math., 3 (1936), 55-88.


\bibitem{v}
J. Van Heijenoort, On locally convex manifolds, Comm. Pure Appl. Math. {\bf 5} (1952), 223-242.

\end{thebibliography}
\end{document}